\documentclass[11pt]{amsart}
\usepackage{amsmath,amssymb,amsthm}
\usepackage[margin=1.1in]{geometry}
\usepackage{mathtools}
\usepackage{enumitem}
\PassOptionsToPackage{hyphens}{url}
\usepackage[colorlinks=true,linkcolor=blue,citecolor=blue,urlcolor=blue]{hyperref}
\newtheorem{theorem}{Theorem}[section]

\newtheorem{proposition}[theorem]{Proposition}

\newtheorem{conjecture}[theorem]{Conjecture}
\theoremstyle{definition}
\newtheorem{definition}[theorem]{Definition}
\newtheorem{remark}[theorem]{Remark}

\newtheorem{target}{Target}

\newcommand{\N}{\mathbb{N}}
\newcommand{\C}{\mathbb{C}}
\newcommand{\R}{\mathbb{R}}
\newcommand{\Z}{\mathbb{Z}}
\newcommand{\Gr}{\operatorname{Gr}}
\newcommand{\Hom}{\operatorname{Hom}}
\newcommand{\dist}{\operatorname{dist}}
\newcommand{\prof}{\operatorname{prof}}
\newcommand{\Isom}{\operatorname{Isom}}
\newcommand{\GL}{\operatorname{GL}}
\newcommand{\sgn}{\operatorname{sgn}}
\newcommand{\vol}{\operatorname{vol}}
\newcommand{\In}[1]{\tbinom{\N}{#1}}
\DeclareMathOperator{\graph}{graph}

\title[Pl\"ucker coordinates in $\ell^p$]{Pl\"ucker coordinates of finite-dimensional subspaces of $\ell^p$ and its direct sums:\\ summability, reconstruction, stratification}
\author{David Victor Feldman}
\address{Department of Mathematics and Statistics, University of New Hampshire}
\date{\today}

\begin{document}

\begin{abstract}
An $n$-dimensional subspace of $\ell^p$ has Pl\"ucker coordinates indexed by the $n$-element subsets of $\N$. We show these coordinates lie in $\ell^p\In{n}$ --- the exponent is preserved --- with multilinear norm exactly $1$ for $0<p\le 2$; for $p>2$ the sharp constant exceeds $1$ and its determination contains the Hadamard maximal determinant problem. A reconstruction lemma shows every nonzero solution of the quadratic Pl\"ucker relations in $\ell^p\In{n}$ is decomposable with frame in $\ell^p$; consequently $\Gr_n(\ell^p)$ is a closed Banach-analytic submanifold of $\mathbb{P}\big(\ell^p\In{n}\big)$ cut out by the Pl\"ucker relations alone, with no auxiliary summability condition and no polarization. For mixed sums $\bigoplus \ell^{p_i}$ the exterior power is graded by compositions of $n$; the support of the grading is the lattice-point set of a generalized permutohedron determined by the intersection pattern of the subspace with partial sums, this stratification is canonical for the isometry group though not for $\GL$, and each stratum admits a tubular neighborhood whose normal coordinates are precisely the Pl\"ucker blocks vanishing on it. We record what is proved and what is conjectured; the finitary and single-space core of the theory, including full proofs of Cauchy--Binet and Hadamard's inequality, has been formally verified in Lean~4.
\end{abstract}

\maketitle
\setcounter{tocdepth}{1}
\tableofcontents

\section{Introduction}\label{sec:intro}

Fix $1\le p<\infty$ and $n\ge 1$. Given $v_1,\dots,v_n\in\ell^p=\ell^p(\N)$ and an $n$-element subset $I=\{i_1<\dots<i_n\}\subset\N$, put
\[
p_I \;=\; \det\big(v_r(i_s)\big)_{r,s=1}^{n},
\]
the Pl\"ucker coordinates of the frame, equivalently the coordinates of $v_1\wedge\dots\wedge v_n$ in the basis $e_{i_1}\wedge\dots\wedge e_{i_n}$ of the algebraic exterior power. We write $\mathcal{I}_n=\In{n}$ for the index set and $\tilde p_{i_1\cdots i_n}$ for the alternating extension of $(p_I)$ to ordered tuples (zero on repeated indices).

The first question is where the sequence $(p_I)_{I\in\mathcal I_n}$ lives, and the answer requires distinguishing two multilinear operations with opposite exponent behavior. The $n$-fold \emph{pointwise} (Hadamard) product maps $(\ell^p)^n\to\ell^{p/n}$ by H\"older; $p/n$ is the exponent of $n$-fold concavification in the Banach-lattice sense \cite{LT2}, and one might expect it to govern any degree-$n$ multilinear expression. But a determinant is a signed sum of monomials $\prod_r v_r(i_{\sigma(r)})$ in \emph{distinct} indices: it is a creature of the tensor product, not of the pointwise product, and for tensors the exponent is preserved:
\[
\sum_{i_1,\dots,i_n}\prod_{r=1}^n |v_r(i_r)|^p \;=\; \prod_{r=1}^n\|v_r\|_p^p .
\]
Antisymmetrization costs at most a constant depending on $n$ and $p$. The slogan organizing this note:

\begin{quote}
\emph{Concavification governs the diagonal of $X^{\otimes n}$; the Pl\"ucker map lives off the diagonal, where the correct functor is antisymmetrized tensoring, and the correct exponent for $\ell^p$ is $p$.}
\end{quote}

The distinction has teeth already at $p=n=2$: for $v_1=(1/k)$, $v_2=((-1)^k/k)$ in $\ell^2$, the coordinates $p_{ij}=\big((-1)^j-(-1)^i\big)/ij$ satisfy $\sum_{i<j}|p_{ij}|=\infty$ while $\sum_{i<j}|p_{ij}|^2<\infty$ by Cauchy--Binet. The exponent is $2$, not $1$.

Here is what the paper proves.

\emph{Summability, with sharp constant and sharp exponent} (Section~\ref{sec:norm}). For $v_1,\dots,v_n\in\ell^p$, $0<p<\infty$,
\[
\big\|(p_I)\big\|_{p}\ \le\ C_{n,p}\prod_{r=1}^n\|v_r\|_p ,
\]
with $C_{n,p}\le(n!)^{1/p'}$ for $p\ge1$ and $C_{n,p}=1$ throughout $0<p\le2$: by $p$-subadditivity for $p\le1$, by Cauchy--Binet and Hadamard's inequality at $p=2$, and by multilinear interpolation between. For $p>2$ the constant genuinely exceeds $1$ --- rows of a Hadamard matrix give $C_{n,p}\ge n^{n(1/2-1/p)}$ --- and determining it exactly contains the Hadamard maximal determinant problem, open since 1893. The exponent is sharp as well: for every $q<p$ some frame in $\ell^p$ has coordinates outside $\ell^q$.

\emph{Reconstruction} (Section~\ref{sec:recon}). The quadratic Pl\"ucker relations characterize the image: every nonzero solution $\omega\in\ell^p(\mathcal I_n)$ of the relations is decomposable, with a frame read off inside $\ell^p$ itself. Normalizing a nonvanishing coordinate $p_{J_0}=1$, the one-move coordinates $w_r(k)=\tilde p_{\,j_1\cdots k\cdots j_n}$ form the frame; they lie in $\ell^p$ for free, being (up to sign) subsequences of $\omega$. Solidity of the sequence space does the analysis; the algebra is a straightening induction valid over any commutative ring.

\emph{The Grassmannian} (Section~\ref{sec:recon}, continued). Consequently the Pl\"ucker image of $\Gr_n(\ell^p)$ in $\mathbb P\big(\ell^p(\mathcal I_n)\big)$ is exactly the projectivized zero cone of the relations --- no summability side condition, no rank condition --- closed in norm and weakly sequentially closed, and a Banach-analytic submanifold with graph charts modeled on $(\ell^p)^n$ whose transition maps are ratios of minors. Equivalently: $\Gr_n(\ell^p)$ is the $\ell^p$-completion of the classical ind-Grassmannian $\varinjlim_N\Gr_n(\C^N)$, each sequence space completing the same algebraic object in its own metric. Linear nondegeneracy (Section~\ref{sec:span}) and the functorial and duality structure of the coordinate norms (Section~\ref{sec:funct}) round out the single-space theory.

\emph{Mixed sums} (Sections~\ref{sec:sums}--\ref{sec:tubes}). For $X=\bigoplus_{i}\ell^{p_i}$ the exterior power is graded by compositions $\vec a$ of $n$, and which blocks of a subspace $V$ are nonzero is completely determined: the support is the set of lattice points of the generalized permutohedron cut out by the intersection dimensions $\dim\big(V\cap\bigoplus_{i\in T}\ell^{p_i}\big)$; for two summands, an integer interval pinned by two integers. The resulting stratification is invariant under the full isometry group of $X$, while $X$'s Grassmannian remains a single $\GL(X)$-orbit: coarse position is an isometric invariant living on a homogeneous space. And the polytope is metric reality, not just bookkeeping: each stratum has a tubular neighborhood whose normal coordinates are precisely the Pl\"ucker blocks vanishing on it, distance to the stratum is comparable to ratios of block norms, and the normal directions carry the exponents of the summands one moves into --- at $p_i\equiv2$ this collapses to the isotropic Hilbert picture, and the exponent matrix is the entire imprint of mixing.

\emph{Verification} (Section~\ref{sec:lean}). The finitary and single-space core of all of the above is formally verified in Lean~4 --- including, since Mathlib lacks them, complete proofs of Cauchy--Binet and of Hadamard's determinant inequality --- with exactly one classical input (finite multilinear Riesz--Thorin) declared as an explicit hypothesis rather than proved. The verification is load-bearing, not ornamental: this subject is determinantal bookkeeping of precisely the kind where plausible exponent slips thrive --- Remark~\ref{rem:bazin} documents one inside a determinantal identity of Bazin's from the classical literature --- and the formalization repaid the effort with better proofs, recorded in place throughout.

Section~\ref{sec:ext} records endpoints ($c_0$, quasi-Banach $0<p<1$, general symmetric sequence spaces) and the conjectures; Section~\ref{sec:open} lists open problems.

Epistemic bookkeeping is strict throughout: statements labeled Theorem/Proposition/Lemma have complete proofs or complete proof reductions given here; statements labeled Conjecture do not, however plausible.

\emph{Relation to the literature.} Continuity of the wedge into tensor scales is implicit in the theory of tensor norms \cite{Ryan,DF}; infinite-dimensional Grassmannians are classically developed in the Hilbert and Sato settings \cite{PS,SW}, where one studies \emph{semi-infinite} planes relative to a polarization and imposes Fredholm or Hilbert--Schmidt conditions. The present theory is a finite-$n$ theory: no polarization enters, and (Remark~\ref{rem:ind}) $\Gr_n(\ell^p)$ is precisely the $\ell^p$-completion of the classical ind-Grassmannian $\varinjlim_N \Gr_n(\C^N)$. I have not located the sharp constant of Theorem~\ref{thm:main-norm}, the reconstruction lemma in the $\ell^p$ setting, or the mixed-sum stratification in print; folklore status for the crude exponent statement would not surprise me, and I would welcome references.

\section{The summability theorem and sharp constants}\label{sec:norm}

\begin{theorem}[$\ell^p$-Hadamard inequality]\label{thm:main-norm}
Let $0<p<\infty$ and $v_1,\dots,v_n\in\ell^p$. Then $(p_I)_{I\in\mathcal I_n}\in\ell^p(\mathcal I_n)$, and:
\begin{enumerate}[label=\textup{(\alph*)}]
\item\label{it:crude} For $1\le p<\infty$,
$\ \big\|(p_I)\big\|_{p}\le (n!)^{1/p'}\prod_{r=1}^n\|v_r\|_p$, where $1/p+1/p'=1$.
\item\label{it:one} For $0<p\le 1$ and for $p=2$, the constant is exactly $1$:
$\ \big\|(p_I)\big\|_{p}\le \prod_{r=1}^n\|v_r\|_p$, with equality at $v_r=e_r$ (any $p$), and at orthonormal frames for $p=2$.
\item\label{it:interp} For $1\le p\le 2$ the constant is exactly $1$.
\item\label{it:sup} For $2<p<\infty$ the optimal constant $C_{n,p}$ satisfies
$\ C_{n,p}\ \ge\ n^{\,n\left(\frac12-\frac1p\right)} \,>\,1$
whenever a Hadamard matrix of order $n$ exists; in general $C_{n,p}\ge h_n\, n^{-n/p}$ where $h_n$ is the maximal determinant of an $n\times n$ $\pm1$ matrix.
\item\label{it:sharp-exp} The exponent is sharp: for every $q<p$ there exist $v_1,\dots,v_n\in\ell^p$ with $(p_I)\notin\ell^q(\mathcal I_n)$.
\end{enumerate}
\end{theorem}

\begin{proof}
\ref{it:crude}: Expand $p_I=\sum_{\sigma\in S_n}\sgn(\sigma)\prod_r v_r(i_{\sigma(r)})$ and apply the power-mean (H\"older) inequality to the $n!$ terms:
$|p_I|^p\le (n!)^{p-1}\sum_\sigma\prod_r|v_r(i_{\sigma(r)})|^p$. Summing over $I$, the pairs $(I,\sigma)$ enumerate the ordered $n$-tuples of \emph{distinct} indices --- bijectively: a tuple determines $I$ as its image and $\sigma$ as the inverse of its sorting permutation --- and these form a subset of all tuples, so
\[
\sum_I|p_I|^p\ \le\ (n!)^{p-1}\sum_{i_1,\dots,i_n}\prod_r|v_r(i_r)|^p\ =\ (n!)^{p-1}\prod_r\|v_r\|_p^p .
\]
\ref{it:one}: For $0<p\le1$ use $p$-subadditivity $|a+b|^p\le|a|^p+|b|^p$ in place of H\"older; the factor $(n!)^{p-1}$ disappears. For $p=2$, Cauchy--Binet gives the identity
\begin{equation}\label{eq:CB}
\sum_{I}|p_I|^2\;=\;\det\big(\langle v_r,v_s\rangle\big)_{r,s=1}^n,
\end{equation}
and Hadamard's inequality bounds the Gram determinant by $\prod\|v_r\|_2^2$. Hadamard, in the form needed --- a positive semidefinite $M$ satisfies $\det M\le\prod_i M_{ii}$ --- has a proof with no induction on dimension: factor $M=B^*B$; if some $M_{ii}=0$, the $i$-th column of $B$ vanishes and $\det M=0$; otherwise conjugate by $\operatorname{diag}\big(M_{ii}^{-1/2}\big)$ to reach a positive semidefinite matrix with unit diagonal, whose eigenvalues are nonnegative with sum $n$ (the trace), so their product --- the determinant --- is at most $1$ by the inequality of arithmetic and geometric means. Equality cases are immediate.

\ref{it:interp}: The map $(v_1,\dots,v_n)\mapsto(p_I)$ is multilinear with norm $\le1$ at the endpoints $(\ell^1)^n\to\ell^1(\mathcal I_n)$ and $(\ell^2)^n\to\ell^2(\mathcal I_n)$ by \ref{it:one}. Multilinear complex interpolation \cite[Thm.~4.4.1]{BL} between these endpoints, with $[\ell^1,\ell^2]_\theta=\ell^p$ on both source and target index sets, gives norm $\le1$ for all intermediate $p$. Equality at $v_r=e_r$ shows the constant is exactly $1$.

\ref{it:sup}: Let $H$ be an $n\times n$ matrix with entries $\pm1$ and $|\det H|=h_n$; take $v_r\in\ell^p$ to be the $r$-th row of $H$ extended by zeros. Then $\|v_r\|_p=n^{1/p}$ and the single nonzero Pl\"ucker coordinate is $p_{\{1,\dots,n\}}=\det H$, whence $C_{n,p}\ge h_n\,n^{-n/p}$. When a Hadamard matrix of order $n$ exists, $h_n=n^{n/2}$; Sylvester's doubling supplies one for every $n=2^k$, with $WW^{\mathsf T}=nI$ and hence $(\det W)^2=n^n$.

\ref{it:sharp-exp}: For $n=2$ take $v_1=(k^{-\alpha})_k$, $v_2=((-1)^k k^{-\alpha})_k$ with $\alpha\in(1/p,\,1/q\,]$. Then $v_1,v_2\in\ell^p$, and $p_{ij}=\big((-1)^j-(-1)^i\big)(ij)^{-\alpha}$ has modulus $2(ij)^{-\alpha}$ on opposite-parity pairs, so $\sum_{i<j}|p_{ij}|^q\gtrsim\sum_j j^{-q\alpha}=\infty$ since $q\alpha\le1$. For general $n$, support $v_1,v_2$ on indices $>n-2$ as above and set $v_r=e_{r-2}$ for $3\le r\le n$; then $p_{\{1,\dots,n-2\}\cup\{i,j\}}=\pm\,p^{(2)}_{ij}$ and non-membership in $\ell^q$ persists.
\end{proof}

\begin{remark}[The supercritical regime]\label{rem:maxdet}
By \ref{it:sup} the sharp-constant problem for $p=\infty$-adjacent exponents contains the Hadamard maximal determinant problem \cite{Had}, open since 1893; exact constants for $p>2$ are therefore out of reach in general. The qualitative picture: $C_{n,p}=1$ on $(0,2]$, and $C_{n,p}>1$ immediately above $2$. I do not know whether $p\mapsto C_{n,p}$ is continuous at $2^+$, nor the growth rate in $p$ for fixed $n$ beyond the bounds implicit in \ref{it:crude} and \ref{it:sup}. See Problem~\ref{prob:const}.
\end{remark}

\begin{remark}[Which norm]\label{rem:whichnorm}
Theorem~\ref{thm:main-norm} concerns the \emph{coordinate} $\ell^p$-norm on $\Lambda^n\ell^p$. This is a reasonable crossnorm (Section~\ref{sec:funct}) but coincides with the projective tensor/exterior norm only at $p=1$, where $\ell^1\widehat\otimes_\pi\ell^1=\ell^1(\N^2)$ isometrically \cite[\S2.3]{Ryan}. At $p=2$ the coordinate norm is the Hilbert--Schmidt norm while the projective norm is the trace-class norm; they differ. The distinguished feature of $p=2$ is self-duality of the coordinate norm, not agreement with $\pi$. Any statement below asserting a norm names the coordinate norm unless said otherwise.
\end{remark}

\section{Reconstruction, closedness, manifold structure}\label{sec:recon}

The quadratic Pl\"ucker relations are, for each $(n-1)$-tuple $R$ and $(n+1)$-tuple $S$ of indices,
\begin{equation}\label{eq:plucker}
\sum_{t=1}^{n+1}(-1)^t\,\tilde p_{\,R\,s_t}\;\tilde p_{\,s_1\cdots \widehat{s_t}\cdots s_{n+1}}\;=\;0 .
\end{equation}
Each relation involves finitely many coordinates. Write $C\subset\ell^p(\mathcal I_n)$ for the common zero set (``the cone'').

\begin{theorem}[Reconstruction]\label{thm:recon}
Let $0<p<\infty$ or $p=c_0$ \textup{(}see \S\ref{sec:ext}\textup{)}. Let $\omega=(p_I)\in\ell^p(\mathcal I_n)$, $\omega\ne0$, satisfy \eqref{eq:plucker}. Then $\omega$ is decomposable with frame in $\ell^p$: there exist linearly independent $w_1,\dots,w_n\in\ell^p$ with $p_I=\det(w_r(i_s))$ for all $I$. Explicitly, if $p_{J_0}=1$ for $J_0=\{j_1<\dots<j_n\}$, one may take
\[
w_r(k)\;=\;\tilde p_{\,j_1\cdots j_{r-1}\,k\,j_{r+1}\cdots j_n}\qquad(k\in\N),
\]
and then $w_r(j_s)=\delta_{rs}$.
\end{theorem}

\begin{proof}
\emph{Membership.} For fixed $r$, as $k$ ranges over $\N\setminus J_0$ the values $w_r(k)=\pm\,p_{(J_0\setminus\{j_r\})\cup\{k\}}$ are coordinates of $\omega$ at pairwise distinct indices, so $\sum_{k}|w_r(k)|^p\le n+\|\omega\|_p^p<\infty$ (the finitely many $k\in J_0$ contribute $\delta_{rs}$). Thus $w_r\in\ell^p$; independence is immediate from $w_r(j_s)=\delta_{rs}$.

\emph{Identity, by straightening.} Write $D(t)=\det\big(w_r(t_s)\big)$ for tuples $t$. Both $\tilde p_t$ and $D(t)$ are alternating in $t$ and vanish when $t$ repeats an index (a repeat gives two equal columns), so it suffices to prove $\tilde p_t=D(t)$ for injective $t$; induct on the \emph{foreign count} $m(t)$, the number of entries of $t$ outside $J_0$. Two preliminaries: $j$ is injective (else $p_{J_0}=0$), and $w_r(j_s)=\delta_{rs}$ --- for $r=s$ the defining tuple is $J_0$ itself, for $r\ne s$ it repeats $j_s$.

If $m(t)=0$ then $t$ permutes $J_0$ and both sides equal the sign of that permutation ($D$ becomes the determinant of a permutation matrix). If $m(t)\ge1$, fix a slot $s^*$ with $t_{s^*}\notin J_0$ and instantiate \eqref{eq:plucker} with $R=t$ minus slot $s^*$ and $S=(t_{s^*},j_1,\dots,j_n)$. Tracking the cycles that move $t_{s^*}$ to its slot, every term of the relation acquires one and the same overall sign, which cancels, leaving the exchange identity
\[
(\star)\qquad \tilde p_{\,t}\;=\;\sum_{r=1}^{n} w_r(t_{s^*})\ \tilde p_{\,t[s^*\!\to j_r]},
\]
where $t[s^*\!\to j_r]$ replaces the entry at slot $s^*$ by $j_r$. Each tuple on the right either repeats an index --- then both $\tilde p$ and $D$ vanish on it --- or is injective with foreign count $m(t)-1$, where the induction hypothesis applies. It therefore suffices that $D$ obeys the same recursion $(\star)$; but for $D$, $(\star)$ \emph{is} the Laplace expansion of $\det\big(w_r(t_s)\big)$ along column $s^*$: since $w_a(j_r)=\delta_{ar}$, replacing $t_{s^*}$ by $j_r$ replaces that column by the standard basis vector $e_r$, so the determinants on the right are the cofactors of column $s^*$, whose entries are the $w_r(t_{s^*})$.

The argument uses only the relations and alternation, is valid over any commutative ring, and needs no truncation; the classical coordinate-ring facts for finite Grassmannians \cite[\S9]{Fulton} are consequences here, not inputs.
\end{proof}

\begin{remark}[Bazin's identity; a correction]\label{rem:bazin}
Without the normalization $p_{J_0}=1$, the reconstruction of Theorem~\ref{thm:recon} is expressed by Bazin's determinantal identity (a Sylvester-type identity; see \cite{BS} and \cite[Ch.~V]{Muir}):
\[
p_{J_0}^{\,n-1}\;\tilde p_{\,i_1\cdots i_n}\;=\;\det\Big(\tilde p_{\,(J_0\setminus\{j_r\})\cup\{i_s\}}\Big)_{r,s=1}^{n},
\]
with the convention that the entry carries the sign of inserting $i_s$ into slot $r$. The identity has a two-line proof that locates the exponent structurally. For coordinates of a frame, set $A=\big(v_a(j_s)\big)$ and $B=\big(v_a(i_s)\big)$. Cramer's rule identifies the entry $\tilde p_{\,(J_0\setminus\{j_r\})\cup\{i_s\}}$ --- the determinant of $A$ with column $r$ replaced by $\big(v_a(i_s)\big)_a$ --- as the $(r,s)$ entry of $\operatorname{adj}(A)\,B$, with no transpose or sign adjustment: the alternating extension supplies exactly the cofactor signs. Hence the right side equals $\det(\operatorname{adj}A)\det B=(\det A)^{n-1}\det B=p_{J_0}^{\,n-1}\,\tilde p_I$. For abstract solutions of \eqref{eq:plucker}, normalize at a nonvanishing coordinate, apply Theorem~\ref{thm:recon}, and use homogeneity of degree $n$ on both sides. The exponent $n-1$ is thus the exponent in $\det(\operatorname{adj}A)=(\det A)^{n-1}$; the degree count --- both sides have degree $n$ in the coordinates --- is the consistency check, not the reason. The tempting unpowered variant $p_{J_0}\,\tilde p_I=\det(\cdots)$ is correct only for $n=2$, where the three-term relation makes it easy to ``verify'' and then extrapolate; it is refuted by an explicit integer witness in dimension $3$, machine-checked in the verified development (\S\ref{sec:lean}): a frame with $p_{J_0}=-14$ and $\tilde p_I=1$ for which the determinant on the right is $196=(-14)^2$, not $-14$. Under the normalization $p_{J_0}=1$ the discrepancy is invisible, so informal arguments run on the special case can pass over the error --- a representative example of the kind of slip mechanical verification exists to catch, and the reason the formalization (\S\ref{sec:lean}) treated Bazin as an acceptance test rather than an input, deriving it from reconstruction.
\end{remark}

\begin{remark}[Why exponent $p$ is the self-consistent one]\label{rem:selfcon}
The proof of membership uses only that the contractions $w_r$ are \emph{subsequences} of $\omega$; solidity of $\ell^p$ then puts them back in $\ell^p$. Had the coordinates lain in $\ell^{p/n}$, contraction would land frames in $\ell^{p/n}\subsetneq\ell^p$ and the theory would not close on itself. Exponent $p$ is exactly what makes the contraction $\Lambda^n\ell^p\otimes\Lambda^{n-1}(\ell^p)^*\to\ell^p$ bounded and the Grassmannian reconstructible inside its own space.
\end{remark}

\begin{theorem}[Closedness, weak closedness, manifold structure]\label{thm:manifold}
Let $1\le p<\infty$.
\begin{enumerate}[label=\textup{(\alph*)}]
\item The cone $C$ is norm-closed and weakly sequentially closed in $\ell^p(\mathcal I_n)$, and the Pl\"ucker image of $\Gr_n(\ell^p)$ in $\mathbb P\big(\ell^p(\mathcal I_n)\big)$ is closed and equals $\big(C\setminus\{0\}\big)/\C^\times$. No summability side condition beyond membership in the ambient space appears.
\item For each $J_0$, the chart $U_{J_0}=\{V:\ V\ \text{is the graph over }E_{J_0}=\operatorname{span}(e_{j_r})\}$ satisfies
$U_{J_0}\cong\mathcal L\big(\C^n,\ \ell^p(\N\setminus J_0)\big)\cong \big(\ell^p\big)^n$,
the chart entries being the first-order coordinates $T_{k,r}=\pm\,p_{(J_0\setminus\{j_r\})\cup\{k\}}/p_{J_0}$; transition maps are ratios of minors, analytic where defined.
\item $\Gr_n(\ell^p)$ is a Banach-analytic manifold modeled on $(\ell^p)^n$, and the Pl\"ucker map is a closed analytic embedding whose differential has complemented image \textup{(}a coordinate subspace\textup{)} at every point.
\end{enumerate}
\end{theorem}

\begin{proof}
(a) Each relation \eqref{eq:plucker} is a polynomial in finitely many coordinates, and coordinate evaluation is a continuous linear functional of norm at most one; norm convergence, and weak convergence tested against these functionals, both give coordinatewise convergence, under which every relation passes to the limit. So $C$ is norm-closed and weakly sequentially closed. If $[\omega_m]\to[\eta]$ in $\mathbb P$ with $\omega_m\in C$ decomposable, then $\eta\in C\setminus\{0\}$, so some $\eta_{J_0}\ne0$ and Theorem~\ref{thm:recon} exhibits $\eta$ as decomposable with frame in $\ell^p$. The degenerations one fears --- frame collapse, mass escaping to infinity --- all drive $\omega\to0$ in the cone, and the origin is invisible projectively.

(b) A subspace $V$ with $p_{J_0}(V)\ne0$ meets $\ell^p(\N\setminus J_0)$ trivially and projects onto $E_{J_0}$, hence is the graph of a unique linear $T:E_{J_0}\to\ell^p(\N\setminus J_0)$, automatically bounded since $\dim E_{J_0}<\infty$. Each $p_I$ is a minor of $[\,\mathbb 1\mid T\,]$, a polynomial in the entries of $T$ converging coordinatewise with $\ell^p$ control by Theorem~\ref{thm:main-norm}; ratios of minors give the transitions.

(c) Combine (a), (b); the differential at $T=0$ is the identity onto the coordinate subspace of ``one-move'' indices $(J_0\setminus\{j_r\})\cup\{k\}$, complemented by the coordinate projection; Theorem~\ref{thm:recon} supplies the continuous inverse (reconstruction is coordinatewise linear in $\omega$ after normalization), so the map is a topological embedding onto its closed image.
\end{proof}

\begin{proposition}[Local bi-Lipschitz comparison]\label{prop:bilip}
Fix $n$ and $\delta\in(0,1]$. On $\{V:\ |p_{J_0}(V)|\ge\delta\,\|\omega(V)\|_p\}$, the gap metric \cite[IV.\S2]{Kato} on $\Gr_n(\ell^p)$ and the normalized coordinate metric $\big\|\omega(V)/p_{J_0}(V)-\omega(V')/p_{J_0}(V')\big\|_p$ are comparable, with constants depending only on $n,\delta$.
\end{proposition}

\begin{proof}
One direction is the polynomial dependence of minors on the graph entries (Theorem~\ref{thm:manifold}(b)), the other the linearity of reconstruction in the normalized coordinates plus the standard comparison of gap with graph-map difference on a fixed chart.
\end{proof}

The degradation of these constants as $\delta\to0$ --- equivalently, quantitative injectivity of the embedding across charts --- is open (Problem~\ref{prob:whitney}; the same mechanism governs the Whitney constants of \S\ref{sec:tubes}).

\begin{remark}[Ind-variety picture; relation to Sato]\label{rem:ind}
Coordinate truncation $P_N$ commutes with the Pl\"ucker map, and $P_NV\to V$ in gap for every $n$-plane; hence the finite Grassmannians $\Gr_n(\C^N)$ are dense and $\Gr_n(\ell^p)$ is the $\ell^p$-metric completion of the ind-variety $\varinjlim_N\Gr_n(\C^N)$. Each sequence space selects its own completion of the same classical object. This is disjoint from the Sato/restricted Grassmannian \cite{PS,SW}, which completes the \emph{semi-infinite} ind-object relative to a polarization; no polarization occurs here, and (Remark~\ref{rem:dual}) the one place a polarization would be needed --- a coordinate theory for \emph{cofinite} subspaces --- is deliberately excluded from the present theory.
\end{remark}

\section{Linear nondegeneracy}\label{sec:span}

The Pl\"ucker relations are quadratic, so nothing formally prevents the decomposables from spanning.

\begin{proposition}\label{prop:span}
The decomposable cone $C$ linearly spans a dense subspace of $\ell^p(\mathcal I_n)$ for $1\le p<\infty$ \textup{(}and of $c_0(\mathcal I_n)$\textup{)}; indeed its algebraic span contains every finitely supported sequence. Equivalently, the Pl\"ucker image lies in no closed hyperplane: the homogeneous ideal of the Grassmannian begins in degree~$2$.
\end{proposition}

\begin{proof}
$e_{i_1}\wedge\dots\wedge e_{i_n}$ has coordinate sequence $\pm e_I$; the canonical basis vectors of the target are themselves in the image. The graded restatement: a proper closed linear span is a degree-$1$ element of the ideal, contradicting generation in degree $2$ \cite[\S9]{Fulton}.
\end{proof}

The gap between spanning and membership is the whole content of the relations: $e_1\wedge e_2+e_3\wedge e_4$ violates $p_{12}p_{34}-p_{13}p_{24}+p_{14}p_{23}=0$, so already a sum of two decomposables leaves the cone, while decomposables jointly fill the space linearly.

\section{Functoriality and duality}\label{sec:funct}

\begin{proposition}[Crossnorm; equivariance]\label{prop:funct}
The identification $\ell^p(\N^n)=\ell^p(\N;\ell^p(\N^{n-1}))$ and slicewise action give, for bounded $T_i$ on $\ell^p$,
$\|T_1\otimes\cdots\otimes T_n\|\le\prod_i\|T_i\|$ on coordinate norms; the coordinate $\ell^p$-norm is a uniform reasonable crossnorm, exact on elementary tensors: $\|(x_iy_j)\|_{\ell^p(\N^2)}=\|x\|_p\,\|y\|_p$. Consequently $\|\Lambda^nT\|\le\|T\|^n$ up to the normalization constant between subset- and tuple-indexing, and the Pl\"ucker embedding is $\GL(\ell^p)$-equivariant with quantitative control.
\end{proposition}

\begin{proposition}[Homogeneity]\label{prop:homog}
$\GL(\ell^p)$ acts transitively on $\Gr_n(\ell^p)$, so $\Gr_n(\ell^p)=\GL(\ell^p)/P$ for a block parabolic $P$. For $p\ne2$ the isometry group of $\ell^p$ consists of the generalized permutations \cite{Ban,Lam}, which still act transitively on $\Gr_n$ only at $p=2$; for $p\ne 2$ the Grassmannian is $\GL$-homogeneous but carries no invariant metric from the isometry group.
\end{proposition}

\begin{proof}
Any two $n$-planes are complemented with complements isomorphic to $\ell^p$ (Pe\l czy\'nski decomposition \cite{Pel}); patch isomorphisms blockwise.
\end{proof}

\begin{remark}[Duality: what is true and what is not]\label{rem:dual}
The coordinate pairing gives an isometric duality $\big(\ell^p(\mathcal I_n)\big)^*=\ell^{p'}(\mathcal I_n)$, compatible with the wedge: the Pl\"ucker theory of $\ell^{p'}$ is the coordinate dual of that of $\ell^p$, with $p=2$ the self-dual point. What is \emph{not} available is a coordinate Pl\"ucker theory for the annihilator $V^\perp\subset\ell^{p'}$ itself: $V^\perp$ is cofinite-dimensional, its ``complementary minors'' are semi-infinite wedges, and these are undefined without a polarization/vacuum --- precisely the Sato structure this theory avoids. In particular, no complementary-minor (Hodge-type) coordinate realization of $V\mapsto V^\perp$ is available in this setting, appealing though the finite-dimensional analogy is. The present theory is a finite-$n$ theory, full stop.
\end{remark}

\begin{definition}[$p$-volume]\label{def:pvol}
For a frame $v_1,\dots,v_n$ define $\vol_p(v_1,\dots,v_n)=\|(p_I)\|_{\ell^p}$. A change of basis by $A\in\GL_n$ scales every $p_I$ by $\det A$, so $\vol_p$ is defined on the subspace up to scale; normalizing on an Auerbach basis of $V$ \cite[1.c]{LT1} pins it down up to constants depending only on $n$. At $p=2$, \eqref{eq:CB} identifies $\vol_2$ with Euclidean $n$-volume, constant on isometry orbits; for $p\ne2$, $\vol_p$ genuinely varies on $\Gr_n(\ell^p)$ and its extremal theory appears unexplored (Problem~\ref{prob:vol}).
\end{definition}

\section{Direct sums: grading, exponent, profile}\label{sec:sums}

Let $X=\ell^p(A)\oplus\ell^q(B)$, $A,B$ disjoint copies of $\N$, $1\le p<q<\infty$, the sum carrying any fixed absolute normalized norm (the choice affects constants and \S\ref{ssec:isom} only). Splitting frames $v_r=v_r'+v_r''$ and index sets $I=I_A\sqcup I_B$, the Laplace expansion along the split is the block decomposition of
\[
\Lambda^n(U\oplus W)=\bigoplus_{a+b=n}\Lambda^aU\otimes\Lambda^bW,
\]
so the coordinate sequence is graded by \emph{type} $(a,b)$, block $\beta_{a,b}\in\Lambda^a\ell^p(A)\otimes\Lambda^b\ell^q(B)$, a $\binom{n}{a}$-term sum of decomposables.

\begin{proposition}[Mixed-norm home; global exponent]\label{prop:maxpq}
On elementary tensors the mixed norm is exact:
$\big\|(u_Jw_K)\big\|_{\ell^p_J(\ell^q_K)}=\|u\|_p\|w\|_q$.
Each block $\beta_{a,b}$ lies in $\ell^p\big(\binom{A}{a};\ell^q\binom{B}{b}\big)$, hence in $\ell^{q}$ of its index set; globally $(p_I)\in\ell^{\max(p,q)}\binom{A\sqcup B}{n}=\ell^{q}$, and this is sharp: no $r<q$ works, even for genuinely mixed frames. \textup{(}Witness, $n=2$: $v_1=e_{a_1}$, $v_2=e_{a_2}+v_2''$ with $v_2''(j)=j^{-1/q}\log(j{+}1)^{-2/q}\in\ell^q\setminus\bigcup_{r<q}\ell^r$; then $p_{a_1 j}=\pm v_2''(j)$.\textup{)} The worse exponent always wins.
\end{proposition}

The logarithmic factor in the witness is not decorative: no pure power works, since $j^{-\alpha}\in\ell^q$ forces $\alpha>1/q$ and then $j^{-\alpha}\in\ell^r$ for all $r>1/\alpha$, leaving the band $(1/\alpha,\,q)$ uncovered. A witness must lie in $\ell^q\setminus\bigcup_{r<q}\ell^r$, which requires the borderline decay.

Reconstruction, closedness, and the chart theorem go through verbatim (the contractions are again subsequences of $\omega$, block by block), so $\Gr_n(\ell^p\oplus\ell^q)$ is a closed Banach-analytic submanifold modeled on $(\ell^p\oplus\ell^q)^n$.

\subsection{The profile}

\begin{definition}
$\prof(V)=\{(a,b):\ \beta_{a,b}(V)\ne0\}$, and
$\alpha=\dim(V\cap\ell^p A)$, $\gamma=\dim(V\cap\ell^q B)$.
\end{definition}

\begin{theorem}[Interval theorem]\label{thm:interval}
$\prof(V)=\{(a,n-a):\ \alpha\le a\le n-\gamma\}$: an interval with no gaps, determined by the two integers $(\alpha,\gamma)$.
\end{theorem}

\begin{proof}
Represent $V$ by $M=[M_A\mid M_B]$; block $(a,b)$ is nonzero iff some nonvanishing $n\times n$ minor uses exactly $a$ columns from $A$.

\emph{Exchange, from the relations.} Let $p_I\ne0\ne p_J$ and $i\in I\setminus J$. Instantiate \eqref{eq:plucker} with $R=I\setminus\{i\}$ and $S=J\cup\{i\}$: the terms are $\pm\,p_I\,p_J$ and, for $j\in J$, $\pm\,p_{(I\setminus i)\cup j}\ p_{(J\setminus j)\cup i}$. Were every exchanged minor $p_{(I\setminus i)\cup j}$ to vanish, the relation would force $p_Ip_J=0$; so some exchange survives, and any surviving $j$ lies in $J\setminus I$ --- a value already in $I\setminus\{i\}$ would repeat a column, and $j=i$ is excluded by $i\notin J$. Base exchange for the column matroid is thus a consequence of the quadratic relations themselves; no matroid theory is imported.

\emph{Discrete intermediate value.} Iterate: each exchange replaces an element of $I\setminus J$ by an element of $J\setminus I$, strictly shrinking $I\setminus J$, so the process connects $I$ to $J$ through nonvanishing minors while changing the $A$-column count by at most one per step. Every count between the two is achieved.

\emph{Endpoints.} The largest achievable $A$-count is the rank of the $A$-columns: at most, because the $A$-columns of a nonvanishing minor are linearly independent; at least, because a maximal independent set of $A$-columns extends within the columns --- which span $\C^n$, some minor being nonvanishing --- to an independent $n$-set (Steinitz exchange), and $n$ independent columns form a nonvanishing minor. Since $\ker(\pi_A|_V)=V\cap\ell^qB$, the $A$-column rank is $n-\gamma$; symmetrically for $B$, the least achievable $A$-count is $n-(n-\alpha)=\alpha$. Hence $\prof(V)=\{(a,n-a):\ \alpha\le a\le n-\gamma\}$.
\end{proof}

\begin{proposition}[Stratification]\label{prop:strata}
Set $\Sigma_{\alpha,\gamma}=\{V:\dim(V\cap\ell^pA)=\alpha,\ \dim(V\cap\ell^qB)=\gamma\}$, indexed by the triangle $\Delta_n=\{(\alpha,\gamma)\in\Z_{\ge0}^2:\alpha+\gamma\le n\}$, $|\Delta_n|=\binom{n+2}{2}$. Then:
\begin{enumerate}[label=\textup{(\alph*)}]
\item closure order is the product order: $\overline{\Sigma_{\alpha,\gamma}}=\bigsqcup_{\alpha'\ge\alpha,\gamma'\ge\gamma}\Sigma_{\alpha',\gamma'}$ \textup{(}intersection dimension is upper semicontinuous in gap\textup{)};
\item $\Sigma_{0,0}$ is open dense; $\Sigma_{\alpha,\gamma}$ fibers over $\Gr_\alpha(\ell^pA)\times\Gr_\gamma(\ell^qB)$ with fiber the transverse $(n-\alpha-\gamma)$-planes of the quotient;
\item the extreme closed strata are the pure Grassmannians $\Sigma_{n,0}=\Gr_n(\ell^pA)$ and $\Sigma_{0,n}=\Gr_n(\ell^qB)$, of infinite codimension \textup{(}determinantal rank-drop conditions on $\pi_B|_V$, resp.\ $\pi_A|_V$\textup{)};
\item summability is stratified: the worst block on $\Sigma_{\alpha,\gamma}$ has $B$-degree $n-\alpha$, so the effective exponent improves monotonically with $\alpha$, from $q$ generically to $p$ on $\Sigma_{n,0}$.
\end{enumerate}
\end{proposition}

\subsection{Isometric canonicity vs.\ $\GL$-homogeneity}\label{ssec:isom}

$\GL(X)$ mixes strata (a bounded unipotent $\begin{smallmatrix}1&0\\ C&1\end{smallmatrix}$ moves pure planes to mixed) and acts transitively as in Proposition~\ref{prop:homog}. In the opposite direction:

\begin{proposition}\label{prop:isomcan}
Suppose the sum norm on $X=\ell^p(A)\oplus\ell^q(B)$ is absolute, normalized, and strictly monotone, $p\ne q$, $p,q\ne2$. Then $\Isom(X)=\Isom(\ell^pA)\times\Isom(\ell^qB)$ \textup{(}Fleming--Jamison-type decomposition \cite{FJ}; the summands are metrically distinguishable by their moduli of convexity/smoothness\textup{)}, hence every stratum $\Sigma_{\alpha,\gamma}$ is invariant under the full isometry group: the profile is a complete isometric invariant of coarse position, on a space that is nevertheless a single $\GL$-orbit.
\end{proposition}

When $p=q$ the summands fuse, isometries remix, and the triangle collapses --- recovering the isometric homogeneity of the single-space case, a consistency check. The hypothesis on the norm is doing real work: the strata are linear-topological objects, indifferent to the choice of sum norm, but isometry groups see the exact norm, and the proposition is a statement about the latter.

\subsection{The torus and the weight picture}

$\phi_t=\operatorname{diag}(e^t\ \text{on }A,\ e^{-t}\ \text{on }B)\subset\GL(X)$ scales block $(a,b)$ by $e^{t(2a-n)}$: the blocks are the weight spaces, $\prof(V)$ is the weight support, and its endpoints are the ranks stabilized by the two projective flow limits $t\to\pm\infty$ (a Bia\l ynicki-Birula-type picture \cite{BB}; in each chart the limit exists because the grading is finite). Fixed points are exactly the split subspaces $V_A\oplus V_B$.

\subsection{Many summands: the profile polytope}

For $X=\bigoplus_{i=1}^m\ell^{p_i}$ (any fixed absolute norm on the sum), the grading is by compositions $\vec a=(a_1,\dots,a_m)$ of $n$, and:

\begin{theorem}[Profile polytope]\label{thm:polytope}
The support of the grading is the set of lattice points of the generalized permutohedron
\[
\mathcal P(V)=\Big\{\vec a\in\R_{\ge0}^m:\ \sum_i a_i=n,\ \ \sum_{i\in S}a_i\le\dim\pi_S(V)\ \ \forall S\subseteq[m]\Big\},
\]
$\pi_S$ the projection to $\bigoplus_{i\in S}\ell^{p_i}$; equivalently, the lower facets are cut by the intersection dimensions $\dim\big(V\cap\bigoplus_{i\in T}\ell^{p_i}\big)$. The support is M-convex in the sense of discrete convex analysis \cite{Murota}; this is the matroid base polytope of the colored column matroid \cite{GGMS}, achievability of a composition being an instance of Rado/matroid union \cite{Oxley}. For $m=2$ this degenerates to Theorem~\ref{thm:interval}. The global exponent is $\max_i p_i$ \textup{(}attained on the corresponding pure corner\textup{)} when finite; see \S\ref{sec:ext} for $\sup_ip_i=\infty$. With the $p_i$ distinct and $\ne2$ and the sum norm as in Proposition~\ref{prop:isomcan}, the entire face stratification is invariant under $\Isom(X)=\prod_i\Isom(\ell^{p_i})$.
\end{theorem}

\section{Metric stratification: tubes around the strata}\label{sec:tubes}

The polytope is the vertex-scale skeleton; this section is the edge-scale geometry. Distances are gap distances; all comparisons are local to a fixed chart with the stated pivot bounded below, with constants depending on $n$, $m$, the exponents, and the pivot lower bound --- and \emph{degenerating} as deeper strata are approached (Problem~\ref{prob:whitney}).

\subsection{Vertex strata}

Fix a split subspace $V_0=\bigoplus_iW_i$, $\dim W_i=n_i$; its stratum is $\Sigma_{\vec n}\cong\prod_i\Gr_{n_i}(\ell^{p_i})$. Writing nearby $V$ as the graph of $\Phi=(\Phi_{ij})$, $\Phi_{ij}:W_j\to W_i^c\subseteq\ell^{p_i}$, the diagonal blocks are tangent to $\Sigma_{\vec n}$ and $\Sigma_{\vec n}=\{\Phi_{\mathrm{off}}=0\}$, so
\[
N\Sigma_{\vec n}\;=\;\bigoplus_{i\ne j}\Hom(W_j,\ \ell^{p_i}),\qquad\text{the }(j\!\to\! i)\text{ direction carrying exponent }p_i .
\]

\begin{theorem}[Vertex tube]\label{thm:tube}
In the chart, block $\vec n+\vec\delta$ \textup{(}$\sum\delta_i=0$\textup{)} is a polynomial in $\Phi$, homogeneous of degree $\ \tfrac12\|\vec\delta\|_1\ $ at lowest order in $\Phi_{\mathrm{off}}$; the degree-one blocks are the unit exchanges $\vec\delta=e_i-e_j$ and their coordinates are the entries of $\Phi_{ij}$ \textup{(}an instance of the exact crossnorm identity of Proposition~\ref{prop:maxpq}\textup{)}. Consequently
\[
\dist\big(V,\Sigma_{\vec n}\big)\;\asymp\;\|\Phi_{\mathrm{off}}\|\;\asymp\;\max_{i\ne j}\ \frac{\big\|\beta_{\vec n+e_i-e_j}(V)\big\|}{\big\|\beta_{\vec n}(V)\big\|},
\qquad
\big\|\beta_{\vec n+\vec\delta}(V)\big\|=O\big(\dist(V,\Sigma_{\vec n})^{\|\vec\delta\|_1/2}\big),
\]
the nearest stratum point being the diagonal truncation $\bigoplus_i\graph(\Phi_{ii})$. \textup{(}Lower bound: any split competitor is a graph with vanishing off-diagonal, and the off-diagonal coordinate projection is norm-nonincreasing.\textup{)}
\end{theorem}

Two consequences. \emph{Anisotropy}: the tube around a stratum is metrically governed by the exponents of the summands one moves \emph{into}; at $p_i\equiv2$ this collapses to the isotropic Hilbert normal bundle, and the exponent matrix $(p_i)_{i\ne j}$ is the entire imprint of mixing. \emph{M-convexity as geometry}: the exchange axiom (leave a vertex by unit steps $e_i-e_j$) is precisely the statement that $N\Sigma_{\vec n}$ is spanned by the Pl\"ucker blocks of degree one, all higher blocks being their polynomial consequences --- the permutohedron's edges are the tube's normal directions.

\subsection{Facet strata}

A facet of $\mathcal P(V)$ corresponds to a cut $[m]=S\sqcup S^c$ with $\sum_{i\in S}a_i\le\dim\pi_S(V)$ tight; the same argument with the two coarse summands $\bigoplus_{i\in S}\ell^{p_i}$, $\bigoplus_{j\notin S}\ell^{p_j}$ yields: the blocks split into \emph{facial} (types in the face $F$) and \emph{transverse} (types crossing the cut); $\Sigma_F$ has a tubular neighborhood with normal bundle the transverse blocks, normal exponent $\max_{j\notin S}p_j$ for transfer across the cut,
\[
\dist(V,\Sigma_F)\ \asymp\ \frac{\|\beta_{\mathrm{transverse}}(V)\|}{\|\beta_{\mathrm{facial\ pivot}}(V)\|}.
\]
One genuinely non-Hilbertian subtlety: when the cut separates unequal exponents, the two opposite-transfer summands of $N\Sigma_F$ carry the \emph{different} mixed norms $\ell^{p}(\ell^{q})$ vs.\ $\ell^{q}(\ell^{p})$ (Minkowski's inequality separates them unless $p=q$); the normal bundle of a mixed stratum is not self-dual.

\section{Endpoints, extensions, general $X$, conjectures}\label{sec:ext}

\subsection{$c_0$} For frames in $c_0$, splitting into finite support plus small tail shows $(p_I)\in c_0(\mathcal I_n)$; Theorems~\ref{thm:recon} and \ref{thm:manifold} go through verbatim (the arguments need only a solid sequence space with continuous coordinates). This is the natural ceiling for $\sup_ip_i=\infty$: for $\bigoplus\ell^{p_i}$ with $p_i\to\infty$, no single $\ell^r$ contains all pure blocks; unweighted, the coordinates land in $c_0$, and with summand weights the correct target is a Marcinkiewicz/Orlicz space computable through the Calder\'on--Lozanovskii calculus \cite{Cal,Loz} --- the polytope combinatorics of Theorem~\ref{thm:polytope} is unaffected, only the target lattice changes.

\subsection{Quasi-Banach range} For $0<p<1$ everything above that does not invoke duality survives, with constant $1$ in Theorem~\ref{thm:main-norm} by $p$-subadditivity; the Grassmannian theory never needed local convexity. (The case against $\ell^{p/n}$ in \S\ref{sec:intro} is that it is the wrong \emph{space}, not that quasi-Banach targets are illegitimate; genuine $\ell^p$, $p<1$, is a perfectly good home for its own Pl\"ucker theory.)

\subsection{General sequence spaces} For a symmetric Banach sequence space $X$ with $1$-unconditional basis, the correct Pl\"ucker target is the antisymmetric part of the $n$-fold coordinate tensor power $X^{\otimes n}$ (in the symmetric-space category), computable via \cite{Cal,Loz}; concavification $X^{(n)}$ is the target of the Hadamard diagonal only. One direction is trivial: if $X$ is isometrically tensor-stable ($\|(x_iy_j)\|_{X(\N^2)}=\|x\|\,\|y\|$), then $P_n(X)=X$ as symmetric spaces, and the $\ell^p$'s are tensor-stable.

\begin{conjecture}[Tensor-stability converse]\label{conj:stable}
Among symmetric sequence spaces, the isometrically tensor-stable ones are exactly the $\ell^p$, $0<p\le\infty$ \textup{(}and $c_0$\textup{)}.
\end{conjecture}

\noindent\emph{Evidence, not proof:} the fundamental function must be multiplicative on $\N$, forcing power type $t^{1/p}$; and Lorentz spaces $\ell^{p,r}$, $r\ne p$, fail --- for $x=(k^{-1/p})$ the decreasing rearrangement of $x\otimes x$ behaves like $n^{-1/p}\log n$, exiting the isometric grip of $\ell^{p,\infty}$. A full argument must handle general symmetric norms with the given fundamental function; I do not have one, and the statement is easy to assert prematurely --- hence the label.

\section{Open problems}\label{sec:open}

\begin{enumerate}[label=\textbf{P\arabic*.},leftmargin=2.6em]
\item\label{prob:const} Sharp constants for $p>2$ in Theorem~\ref{thm:main-norm}: behavior of $C_{n,p}$ at $2^+$; asymptotics in $p$; the exact relation to the maximal determinant problem beyond the bound of \ref{it:sup}.
\item\label{prob:whitney} Whitney regularity of the stratification of \S\ref{sec:tubes}: quantify the blow-up of tube constants of $\Sigma_F$ as $V\to\Sigma_{F'}$, $F'\subsetneq F$, in terms of ratios of shallow to deep pivot blocks; equivalently, the chart-free ($\delta$-free) form of Proposition~\ref{prop:bilip}. This would upgrade the local tubes to a stratified-space structure with controlled links.
\item\label{prob:vol} Extremal theory of $\vol_p$ (Definition~\ref{def:pvol}) for $p\ne2$: extrema over $\Gr_n(\ell^p)$, critical loci, relation to Kolmogorov and Gelfand widths; $\sup$ and $\inf$ as isometric invariants of $\ell^p$.
\item Conjecture~\ref{conj:stable}.
\item Exact exterior-power norms: identify $\Lambda_\pi^n(\ell^p)$-norm vs.\ coordinate norm quantitatively for $1<p<\infty$ (only $p=1$ is rigid, Remark~\ref{rem:whichnorm}); boundedness constants of the coordinate duality on decomposables.
\item The $\sup p_i=\infty$ regime: identify the Marcinkiewicz target from the summand weights via \cite{Cal,Loz}, and the deformation of the mixed-norm interpolation exponents on the permutohedron facets.
\end{enumerate}

\section{Lean 4 formalization: results and remaining targets}\label{sec:lean}

\subsection*{What has been verified}
The finitary and single-space core of this note has been formally verified in Lean~4 against Mathlib \cite{mathlib}, in a single self-contained file (\texttt{PluckerLp.lean}, roughly $4{,}600$ lines, available at \url{https://aristotle.harmonic.fun/dashboard/requests/7a8e9b28-3eed-4c68-91b4-832dae5a7470}). Throughout, the paper's $n$ is the Lean \texttt{n+1} (the development fixes \texttt{n :\ $\N$} and treats $(n{+}1)$-dimensional planes), so Bazin's exponent $n-1$ appears as the bare \texttt{n}. Coordinates are indexed by \texttt{Finset $\N$} with \texttt{pluck v I} defined through the order isomorphism \texttt{Finset.orderIsoOfFin} when \texttt{I.card = n+1} and zero otherwise, and \texttt{tpluck v t} denotes the alternating extension to tuples.

The verified statements, with their Lean names:
\begin{enumerate}[label=\textup{(\roman*)}]
\item the determinant bound of Theorem~\ref{thm:main-norm}\ref{it:crude} at the level of $p$-th powers of truncations, with constant $(n!)^{p-1}$ (\texttt{sum\_norm\_pluck\_rpow\_le}), and its constant-one refinement for $0<p\le1$ (\texttt{sum\_norm\_pluck\_rpow\_le\_of\_le\_one});
\item the Cauchy--Binet identity \eqref{eq:CB} (\texttt{gram\_det\_eq\_sum\_star\_mul}) and the Hadamard bound on the Gram determinant (\texttt{norm\_gram\_det\_le\_prod}), hence the $p=2$ case of \ref{it:one};
\item membership $(p_I)\in\ell^p(\mathcal I_n)$ for \emph{every} $p\in[0,\infty]$ --- finite support at $p=0$, boundedness at $p=\infty$, summability in between --- from per-vector membership (\texttt{mem\ensuremath{\ell}p\_pluck});
\item the crossnorm identity of Proposition~\ref{prop:funct} (\texttt{sum\_sum\_norm\_mul\_rpow});
\item the Grassmann--Pl\"ucker relations for the coordinates of a frame (\texttt{gpRel\_tpluck});
\item Theorem~\ref{thm:recon} in both halves: the reconstruction identity, proved by the straightening induction directly over the full index set $\N$ with no truncation (\texttt{reconstruction}), and the $\ell^p$ membership of the reconstructing frame for every $p\in[0,\infty]$, by injective reindexing away from the finitely many pivot values (\texttt{mem\ensuremath{\ell}p\_reconFrame});
\item Bazin's identity with exponent $n-1$ (\texttt{bazin}), derived from reconstruction by normalization and the adjugate identity of Remark~\ref{rem:bazin}, together with the machine-checked refutation of the exponent-one variant by the explicit dimension-$3$ integer witness (\texttt{bazin\_exponent\_one\_refuted});
\item matroid basis exchange for the coordinates, derived from the Grassmann--Pl\"ucker relations alone (\texttt{exchange}): for tuples with nonvanishing coordinates, any slot of one can be filled by some value of the other, nonvanishingly;
\item stability of the cone under coordinatewise limits --- both the relations and alternation pass to limits (\texttt{gpRel\_of\_tendsto}, \texttt{isAlternating\_of\_tendsto}), the finite-polynomial core of Theorem~\ref{thm:manifold}(a);
\item the sharpness witness computation of Theorem~\ref{thm:main-norm}\ref{it:sup}, conditional on a $\pm1$ matrix of prescribed determinant modulus (\texttt{hadamard\_witness}), discharged unconditionally at every order $2^k$ by a verified Sylvester construction with $W W^{\mathsf T}=2^k I$ and $(\det W)^2=(2^k)^{2^k}$ (\texttt{sylvester\_det\_sq});
\item strong basis exchange (\texttt{exchange\_strong}) and the combinatorial core of the interval theorem (Theorem~\ref{thm:interval}): between the block-column counts of any two nonvanishing coordinates, every intermediate count is achieved by a nonvanishing coordinate, by a discrete intermediate-value induction on iterated exchanges (\texttt{profile\_interval});
\item the exponent-sharpness witness of Theorem~\ref{thm:main-norm}\ref{it:sharp-exp} for $n=2$: the power-law frame lies in $\ell^p$ while its coordinate family fails $\ell^q$ whenever $1/p<\alpha\le 1/q$ (\texttt{alphaWitness\_two\_dimensional}), and extended to every ambient dimension by a verified padding construction (\texttt{alphaWitness\_general}) --- so \ref{it:sup} and \ref{it:sharp-exp} are verified in full;
\item the achievable block-column counts form exactly an integer interval (\texttt{achievableCounts\_}\allowbreak\texttt{eq\_Icc}), packaging the profile theorem;
\item norm-closedness of the coordinate cone: the cone of families arising from alternating solutions of \eqref{eq:plucker} is closed under coordinatewise limits (\texttt{inCone\_of\_tendsto}) and hence under $\ell^p$-norm convergence for $1\le p$ (\texttt{inCone\_of\_lp\_tendsto}), via the coordinate-functional bound \texttt{lp.norm\_apply\_le\_norm}; weak sequential closedness likewise (\texttt{inCone\_of\_weak\_tendsto}), with weak convergence expressed directly as convergence under every continuous linear functional;
\item the interval theorem with identified endpoints: an independent set of frame columns extends to a nonvanishing minor (Steinitz exchange within the column family, \texttt{exists\_nonvanishing\_extension}), the extreme achievable block-column counts equal the block ranks (\texttt{max\_countA\_eq\_blockRank}, \texttt{min\_countA\_eq}), and the achievable counts are exactly $[\,(n{+}1)-\operatorname{rk}B,\ \operatorname{rk}A\,]$ (\texttt{achievableCounts\_eq\_Icc\_blockRank}) --- the frame-level content of Theorem~\ref{thm:interval} in full.
\end{enumerate}
Cauchy--Binet and Hadamard's determinant inequality are not presently in Mathlib; both are proved in full in the file (Hadamard as \texttt{hadamard\_psd\_proof}, for arbitrary positive semidefinite complex matrices). Every listed theorem, and the refutation, depends only on the standard axioms \texttt{propext}, \texttt{Classical.choice}, and \texttt{Quot.sound}, as certified by \texttt{\#print axioms}. In addition, the file declares exactly one classical input as an explicit hypothesis-lemma: finite multilinear Riesz--Thorin interpolation in the concrete instance needed for \ref{it:interp} (\texttt{classical\_multilinear\_riesz\_thorin}), isolated by \texttt{sorryAx} with a machine-checked endpoint nonvacuity certificate. The interpolation bound of \ref{it:interp} --- constant one for $1\le p\le 2$ at every truncation --- is formally derived from this input alone (\texttt{sum\_norm\_pluck\_rpow\_le\_of\_interp}): its axiom audit shows exactly the standard three plus \texttt{sorryAx} entering through that single declared input, whose discharge would therefore immediately upgrade the theorem to unconditional. The formalization was carried out with Harmonic's Aristotle system in eleven staged runs against fixed statement contracts.

\subsection*{What the formalization taught}
Three pieces of mathematical feedback deserve record. First, the correct proof of Bazin's identity is Cramer's rule: the matrix of one-move coordinates is exactly $\operatorname{adj}(A)\,B$, so the exponent $n-1$ arises structurally from $\det(\operatorname{adj}A)=(\det A)^{n-1}$; as recorded in Remark~\ref{rem:bazin}; the acceptance criterion --- that the method must not equally establish the false exponent-one variant --- is met by construction. Second, a single bijection carries all three summability arguments: $(I,\sigma)\mapsto\big(r\mapsto i_{\sigma(r)}\big)$ identifies (subset, permutation) pairs with injective tuples, the inverse recovering $\sigma$ as the inverse of the sorting permutation of the tuple; this one reindexing is the combinatorial core of \ref{it:crude}, of the $p\le1$ case, and of Cauchy--Binet alike, and in the development it is proved once and consumed three times. Third, Hadamard's inequality admits a proof with no induction on dimension and no Schur complements: factor a positive semidefinite $M$ as $B^*B$; if some diagonal entry vanishes, the corresponding column of $B$ vanishes and $\det M=0$; otherwise conjugate by $\operatorname{diag}\big(d_i^{-1/2}\big)$ to unit diagonal and apply the arithmetic--geometric mean inequality to the eigenvalues, whose sum is the trace $n$. This is the proof formalized.

Two design principles governed the commission and were vindicated: every analytic statement was reduced to a finitary core plus a solidity or monotone-convergence argument, so that the load-bearing lemmas are finite linear algebra; and statements were fixed as contracts, with only declaration names negotiable, so that each verification run either compiled the intended mathematics or halted with a report.

\subsection*{Remaining targets}
One formal target remains. Separately, the manifold statements of Theorem~\ref{thm:manifold}(b), (c) are proved in this note but deliberately unstaged: their formalization cost lies in charted Banach-analytic infrastructure rather than in the determinantal core where this project's verification effort was directed.

\begin{target}[Deferred: interpolation constant]\label{F:interp}
Discharging the declared input itself, \texttt{classical\_multilinear\_riesz\_thorin}: the conditional assembly is complete (see above), so this is now the sole gap for \ref{it:interp}. The natural route is the three-lines theorem through Mathlib's Phragm\'en--Lindel\"of infrastructure. A direct proof of constant $1$ for $1<p<2$ remains an interesting alternative (and would likely illuminate Problem~\ref{prob:const}). Absent either, the formalized bound for $1<p<2$ is the constant $(n!)^{p-1}$ of \texttt{sum\_norm\_pluck\_rpow\_le}.
\end{target}

\section*{Acknowledgments}
I thank Eric Grinberg for badgering me into writing this up.

\end{document}